\documentclass[reqno,12pt]{amsart}

\usepackage{graphicx, amsmath, amssymb, amscd, amsthm, euscript, psfrag, amsfonts,bm}

\theoremstyle{plain}

\theoremstyle{plain}

\newtheorem{thm}[equation]{Theorem}

\newtheorem{rmk}[equation]{Remark}
\newtheorem{cor}[equation]{Corollary}
\newtheorem{prop}[equation]{Proposition}

\newtheorem{lem}[equation]{Lemma}

\numberwithin{equation}{section}

\newcommand{\sm}{\left(\begin{smallmatrix}}
\newcommand{\esm}{\end{smallmatrix}\right)}
\newcommand{\bpm}{\begin{pmatrix}}
\newcommand{\ebpm}{\end{pmatrix}}

\newcommand{\C}{\mathbb{C}}
\newcommand{\G}{\mathbb{G}}
\newcommand{\N}{\mathbb{N}}
\newcommand{\Q}{\mathbb{Q}}

\newcommand{\R}{\mathbb{R}}
\newcommand{\Z}{\mathbb{Z}}

\newcommand{\bH}{\mathbb{H}}

\newcommand{\bsl}{\backslash}

\newcommand{\op}{\operatorname}
\newcommand{\BR}{\op{BR}}\newcommand{\BMS}{\op{BMS}}\newcommand{\SO}{\op{SO}}
\renewcommand{\G}{\Gamma}\newcommand{\PS}{\op{PS}}\newcommand{\PSL}{\op{PSL}}\newcommand{\T}{\op{T}}
\newcommand{\ba}{\backslash}
\newcommand{\e}{\epsilon}
\newcommand{\z}{\Z}\newcommand{\br}{\R}
\renewcommand{\c}{\C}
\newcommand{\la}{\langle}
\newcommand{\ra}{\rangle}
\newcommand{\s}{\bf s}

\renewcommand{\S}{\mathcal S}

\begin{document}

\title[Equidistribution of horocycles]
{Effective Equidistribution of closed horocycles for geometrically finite surfaces}

\author{Min Lee and Hee Oh}

\address{Mathematics department, Brown university, Providence, RI}\email{minlee@math.brown.edu}
\address{Mathematics department, Brown university, Providence, RI
and Korea Institute for Advanced Study, Seoul, Korea}\email{heeoh@math.brown.edu}

\thanks{The authors are respectively  supported in parts by Simons Fellowship and by NSF Grant \#1068094.}

\begin{abstract}
For a complete hyperbolic surface whose fundamental group
is finitely generated and has critical exponent bigger than $\tfrac 12$, we obtain an effective equidistribution of closed horocycles in
 its unit tangent bundle. This extends a result of Sarnak in 1981 for surfaces of finite area.
We also discuss applications in Affine sieves.
\end{abstract}

\maketitle

\section{Introduction}
Let $G=\PSL_2(\br)$ be the group of orientation preserving isometries of the
hyperbolic plane $\bH^2=\{x+iy: y>0\}$. Let $\G$ be a finitely generated discrete torsion-free
subgroup of $G$, which is not virtually cyclic. 
The quotient space $X:=\G\ba G$ can be identified with the unit tangent bundle of the hyperbolic surface $\G\ba \bH^2$.
 For $x\in \br$ and $y>0$, we
define $$n_x:=\begin{pmatrix} 1& x\\ 0 & 1\end{pmatrix} \quad\text{and} \quad a_y:= \begin{pmatrix} \sqrt y &0 \\  0 & {\sqrt y}^{-1}
                                                                                 \end{pmatrix} .$$
Via the multiplication from the right, the action of $a_y$ on $X$ corresponds to the geodesic flow 
 and the orbits of the subgroup $N:=\{n_x :x\in \br\}$ give rise to the stable horocyclic foliation of  $X$.
For a fixed closed horocycle $[g]N$ in $X$, we consider a sequence of
closed horocycles $[g]Na_y$ as $y\to 0$.


We denote by $\Lambda(\G)$ the limit set of $\G$, the set of accumulation points of
an orbit of $\G$ in the boundary $\partial (\bH^2)=\br \cup\{\infty\}$. 
A point $\xi\in \Lambda(\G)$ is called a parabolic limit point for $\G$ if $\xi$ is the unique fixed point
in $\partial (\bH^2)$ of an element of $\Gamma$.

A horocycle in $\bH^2$ is simply a Euclidean circle tangent to a point in $\partial(\bH^2)$, called the basepoint;
here a circle tangent to $\infty$ is understood as a horizontal line.
 Topological behavior of a horocycle in $\G\ba \bH^2$ is completely determined by its basepoint, say, $\xi$; 
it is closed and non-compact (resp. compact)
if and only if $\xi$ lies outside $\Lambda(\G)$ (resp. $\xi$ is a parabolic limit point of $\G$) \cite{Dalbo}.

When $\G$ is a lattice in $G$, 
$\Lambda(\G)$ is the entire boundary and hence a closed horocycle is necessarily based at a parabolic limit point and compact. In this case,
 Sarnak \cite{Sa} obtained a sharp effective equidistribution: a sequence of
 closed horocycles becomes equidistributed with respect to the $G$-invariant measure 
in $X$ as their length tends to infinity. One can also use the mixing of the geodesic flow via thickening argument
to obtain an effective equidistribution in this case, which yields
less sharp result than Sarnak's. This argument goes back to the 1970 thesis of Margulis \cite{Mt}
and was used by Eskin and McMullen \cite{EM}.
 
When $\G$ is not a lattice, there are always noncompact closed horocycles in $\G\ba \bH^2$, and compact
horocycles exist only if $\G$ contains a parabolic element, or equivalently if there is a cusp in $\G\ba \bH^2$.
Roblin \cite{Ro} showed that 
for a fixed compact piece $N_0$ of $N$, 
the sequence $[g]N_0a_y$  becomes equidistributed in $X$, as $y\to 0$, with respect to an infinite locally
finite measure, called the Burger-Roblin measure (corresponding to the stable horocyclic
foliation).  In \cite{KO1}, it was observed that 
 the relevant dynamics happens only within a compact
part of $N$ even for a noncompact closed horocycle,
 and hence Roblin's theorem also applies to infinite closed horocycles.
But Roblin's proof is non-effective.

\medskip
\noindent{\bf Effective equidistribution of a closed horocycle:}
 Let $0<\delta\le 1$ denote the critical exponent of $\G$, which is also equal to the Hausdorff dimension of
$\Lambda(\G)$. We have $\delta=1$ if and only if $\G$ is a lattice.

Our main goal is to describe an 
effective equidistribution of a closed horocycle in $\G\ba G$.
In the rest of the introduction, we assume that $1/2 < \delta <1$ 
and that the horocycle $[e] N$ is closed in $\G\ba G$.

Set $k_\theta =
 \bpm \cos\theta & \sin\theta \\ -\sin\theta & \cos\theta\ebpm$, $K=\{k_\theta  : 0\le \theta<\pi\}$ and $A=\{a_y: y>0\}$.
We have the Iwasawa decomposition $G=NAK$: any element of $G$ can be uniquely written as
 $n_xa_yk_\theta$ for $n_x\in N, a_y\in A, k_\theta \in K$.
A Casimir operator of $G$ is given as follows:
$$\mathcal C=-y^2\left(\frac{ \partial^2}{\partial x^2} +
\frac{ \partial^2}{\partial y^2} \right) + y \frac{\partial ^2}{\partial x \partial \theta}.$$
For $K$-invariant functions on $G$, $\mathcal C$ acts as
 the (negative of the) hyperbolic Laplacian: $$\Delta =-y^2\left(\frac{ \partial^2}{\partial x^2}+
\frac{ \partial^2}{\partial y^2} \right).$$

By the results of Patterson \cite{Pa}, and of Lax and Phillips \cite{LP}, the Laplace operator $\Delta$ on $L^2(\G\ba \bH^2)=L^2(\G\ba G)^K$
 has only finitely many eigenvalues $$0<\alpha_0=\delta(1-\delta)<\alpha_1\le \cdots
\le \alpha_k <1/4 $$ lying below the continuous spectrum $[1/4,\infty)$. The existence of
a point eigenvalue is the precise reason that our main theorem is stated only for $\delta >1/2$.
Writing $\alpha_1=s_1(1-s_1)$, a positive number
$$0<{\s}_\G <\delta-s_1$$ will be called a spectral gap for $\G$.

The base eigenvalue $\delta(1-\delta)$ is simple and moreover
there exists a {\it positive} eigenfunction $\phi_0\in L^2(\G\ba G)^K$
with $\Delta\phi_0= \delta(1-\delta)\phi_0$. Patterson gave an explicit formula:
$$\phi_{0}(n_xa_yk_\theta) = 
\int_{\Lambda(\G)} \left(\frac{(u^2+1)y}{(x-u)^2 +y^2}\right)^{\delta}   d\nu_i(u)$$
where $\nu_i$ is the Patterson measure on $\Lambda(\G)$ associated to $i\in \bH^2$.
 We normalize $\nu_i$ so that $\|\phi_0\|_2=1$.

Denote by $V$ the unique $G$-subrepresentation of $L^2(\G\ba G)$ on which $\mathcal C$ acts by the scalar 
$\delta(1-\delta)$. The $K$-fixed subspace of $V$ is spanned by $\phi_0$ and
 $V$ decomposes into the orthogonal sum $\oplus_{\ell\in \z} \c \phi_\ell$
where $\phi_\ell\in C^\infty(\G\ba G)$ satisfies
$\phi_\ell(gk_\theta)=e^{2\ell i \theta} \phi_\ell(g)$ for all $g\in \G\ba G$ and
 $k_\theta\in K$ and has unit norm: $\|\phi_\ell\|_2=1$.

We show that there exists $c_{\phi_\ell}\ne 0$ such that for all $y>0$,
$$\int_{(N\cap \Gamma)\ba N} \phi_\ell (n_xa_y) dx = c_{\phi_\ell} \cdot y^{1-\delta} .$$
 For each $\ell\in \z_{\ge 0}$, we have
$$\frac{c_{\phi_{\pm \ell}}}{c_{\phi_0}}= \frac{\sqrt{\Gamma(\delta)\Gamma(\ell +1-\delta)}}{\sqrt{\Gamma(1-\delta)
\Gamma(\delta+\ell)}},$$
up to a unit, where $\G(x):=\int_0^\infty e^{-t} t^{x-1} dt$ denotes the Gamma function  for $x>0$.



The inner product $\la \psi_1, \psi_2 \ra$ in $L^2(\G\ba G)$ is given by
$$\la \psi_1, \psi_2 \ra=\int_{\G\ba G}\psi_1(g)\overline{\psi_2(g)} dg$$ where $dg$ denotes a $G$-invariant measure on $\G\ba G$.
\begin{thm}\label{main2} Suppose that 
$1/2 < \delta <1$. For $\psi\in C_c^\infty\left(\Gamma\bsl G\right)$, as $y\to 0$,
$$ \int_{(N\cap \Gamma)\ba N} \psi(n_x a_y ) \, dx =
 \sum_{\ell \in \z } c_{\phi_\ell} \cdot  \la \psi, \phi_\ell\ra
 \cdot
y^{1-\delta}   +O(\S_3(\psi) y^{1-\delta +\tfrac{2{\s}_\G}5}) $$
where $\sum_{\ell\in \Z}|c_{\phi_\ell} \cdot \la   \psi ,\phi_\ell\ra| \ll \mathcal S_2(\psi).$
Here $\mathcal S_m(\psi)$ denotes the Sobolev norm of $\psi$ of degree $m$.
\end{thm}

\begin{rmk}
\rm
\begin{enumerate}
 \item  We explicitly compute $\phi_{\pm \ell}$, up to a unit, (Theorem \ref{l:rasing-lowering}):
\begin{multline*}
\phi_{\pm \ell}(n_xa_yk_\theta) =\\
e^{\pm 2\ell i\theta} \tfrac{\sqrt{\Gamma(\delta+\ell)\Gamma(1-\delta)}}{\sqrt{\Gamma (\delta) \Gamma(1-\delta+\ell)}}
\int_{\Lambda(\G)}  \left(\frac{(x-u)\mp iy}{(x-u)\pm iy}\right)^{\ell}  \left(\frac{(u^2+1)y}{(x-u)^2 +y^2}\right)^{\delta} d\nu_i(u).
\end{multline*}

\item We note that $\phi_{-\ell}c_{\phi_{-\ell}}=\overline{ \phi_\ell c_{\phi_\ell}}$  for all $\ell \in \Z$.
Indeed, this is an important observation which clarifies a point that
 the main term in Theorem \ref{main2} is a real number for a real-valued function $\psi$.
\item For
smooth functions on the surface $\G\ba \bH^2$, an effective result 
 was obtained in \cite{KO1} and our proof follows the same general strategy
but working with all different $K$-types of base eigenfunctions $\phi_\ell$'s 
 as opposed to studying only the trivial $K$-type $\phi_0$.
\item In \cite{LO}, we obtain an effective equidistribution of
closed horospheres in the unit tangent bundle of hyperbolic $3$ manifold $\Gamma\ba \bH^3$
when the critical exponent of $\Gamma$ is between $1$ and $2$ and 
use this result for an effective counting of circles in Apollonian circle packings.

\end{enumerate}
\end{rmk}

\medskip

Define the measure $\tilde m^{\BR}_{N}$ on $ G$ in the Iwasawa coordinates $G=KAN$:
for $\psi\in C_c( G)$,
$$\tilde m^{\BR}_N(\psi)=
\int_{KAN}\psi(k a_yn_x)y^{\delta -1} dx dyd\nu_i(k(0)) .$$

This measure is left $\G$-invariant and right $N$-invariant,
and the Burger-Roblin measure $m^{\BR}_{N}$ (associated to the stable horospherical subgroup $N$) 
is the measure on $\G\ba G$ induced from $\tilde m^{\BR}_N$.
 The Burger-Roblin measure 
is an infinite measure whenever $0<\delta<1$ \cite{OS} and coincides with a Haar measure
when $\delta=1$.

Theorem \ref{main2} can also be stated as follows:
\begin{thm}\label{main3} Let $1/2 <\delta\le 1 $.
For any $\psi\in C_c^\infty\left(\Gamma\ba G\right)$, as $y\to 0$,
\begin{equation*} \int_{(N\cap \Gamma)\ba N} \psi(n_x a_y ) \, dx =
 \kappa_\G  \cdot m_N^{\BR}(\psi)\cdot y^{1-\delta} 
  +O(\mathcal S_3(\psi)\cdot y^{(1-\delta)+\tfrac{2{\bf s_\G}}5}) \end{equation*}
where
$\kappa_\G=\tfrac{\sqrt{\pi}\G(\delta-\tfrac 12)}{\Gamma(\delta)}
 \cdot \int_{\textsc{$(N\cap \G)\ba N$}}{(x^2+1)^{\delta}} {d\nu_i(x)} >0 $. 
\end{thm}



\medskip

\noindent{\bf Effective orbital counting on a cone:} 
Let $Q$ be a ternary indefinite quadratic form over $\Q$  and $v_0\in \Q^3$ be a non-zero vector
such that $Q(v_0)=0$. 

Let $G_0:=\SO_Q(\br)^\circ$ and $\G< G_0(\z)$ be a finitely generated subgroup
with $\delta>1/2$. 
 For a square-free integer $d$, 
consider the subgroup of $\G$ which stabilizes $v_0$ mod $d$:
$$\G_d:=\{\gamma\in \G: v_0\gamma \equiv v_0 \mod d\}.$$

To define a sector in the cone $\{Q=0\}$, 
let $\iota:\PSL_2(\br)\to G_0$ be an isomorphism such that $\iota(N)=\{g\in G_0: v_0g=v_0\}$.
Fix a norm $\|\cdot \|$ on $\br^3$. For any subset $\Omega\subset K$ and $T>0$,
define the sector
$$S_T(\Omega):=\{v\in v_0A\Omega : \|v\|<T\} .$$
By a theorem of Bourgain, Gamburd and Sarnak \cite{BGS}, there exists
a spectral gap, say ${\s}_0$,  uniform for all $\G_d$, $d$ square free. 
\begin{thm}\label{ec2m}
Suppose that $\Omega$ has only finitely many connected components.
 As $T\to \infty$, we have $$\#\{v\in v_0\Gamma_d\cap S_T(\Omega)\}=
\frac{ \Xi_{v_0}(\G, \Omega)}{{[\G:\G_d]}}  \cdot  T^\delta + O(T^{\delta-\tfrac{4{\s}_0}{55}}).$$
\end{thm}

 Identifying $\G$ with its pull back in $\PSL_2(\br)$, 
$\Xi_{v_0}(\G, \Omega)$ is given by \begin{equation*}\label{xid} \Xi_{v_0}(\G, \Omega):= \tfrac{\sqrt \pi \Gamma(\delta-\tfrac 12)}{\delta\cdot \Gamma(\delta)}
 \int_{(N\cap\G)\ba N} {(1+x^2)^{\delta}} {{d\nu_i(x)}}
\int_{k\in \Omega^{-1}} \frac{ d\nu_i(k(0))} {\|v_0k^{-1} \|^{\delta}} . \end{equation*}
As $\nu_i$ is supported on the limit set $\Lambda(\G)$, $\Xi_{v_0}(\G, \Omega)>0$ if and only if
the interior of $\Omega^{-1}(0)$ intersects $\Lambda(\G)$.

\begin{rmk}\rm  Theorem \ref{ec2m} is proved in \cite{OS} without an error term.
When the norm is {\it $K$-invariant} and $\Omega=K$, Theorem \ref{ec2m} was proved in \cite{KO1}. 
\end{rmk}

 Using the affine linear sieves developed by Bourgain, Gamburd and Sarnak \cite{BGS2}, this theorem has an application in studying almost prime
vectors in the orbit of $\Gamma$, lying in a fixed sector. To illustrate this, consider the quadratic form
 $Q(x_1, x_2, x_2):=x_1^2+x_2^2-x_3^2$ so that
the integral points in the cone $Q=0$ are Pythagorean triples.
 Let $F(x_1, x_2, x_3):=x_3$ be the hypotenuse. 

\begin{thm}
Suppose that the interior of $\Omega^{-1}(0)$ intersects $\Lambda(\G)$. Then
there exists $R>0$ (depending on ${\s}_0$)  such that
 \begin{multline*} \#\{(x_1,x_2, x_3) \in v_0\G \cap S_T(\Omega): \text{$x_3$ has at most $R$ prime factors} \}
 \asymp \frac{T^\delta}{\log T}\end{multline*} 
where $f(T)\asymp g(T)$ means that their ratio is between two positive constants uniformly for all $T\gg 1$.
\end{thm}

The constant $R$ can be computed explicitly
if $\delta$ is sufficiently large, using the work of Gamburd \cite{Gamburd}; for instance, $R=14$ if $\delta>0.9992$.
This has been carefully worked out in \cite{KO1} for Euclidean norm balls and the same analysis works
for sectors, using Theorem \ref{ec2m}.

\bigskip
\noindent{\bf Acknowledgment:} We thank Peter Sarnak for useful comments on the preliminary version of this paper.

\section{Base eigenfunctions of different $K$-types}
Let $G=\PSL_2( \R)$ and let
	$N=\{n_x:x\in \br \},\; A=\{a_y: y>0\}, \; K=\{k_\theta: \theta\in [0,\pi)\}$ with $n_x, a_y, k_\theta $ defined
as in the introduction.
Throughout the paper, let $\Gamma< G$ be a torsion-free discrete finitely generated subgroup with critical exponent $\tfrac{1}{2} <\delta<1$. 

Consider the Casimir operator
$\mathcal C$  given by
$$\mathcal C= -y^2\left(\frac{ \partial^2}{\partial x^2} +\frac{ \partial^2}{\partial y^2} \right) +y \frac{\partial ^2}{\partial x \partial \theta}.$$

By Lax-Phillips \cite{LP}, $L^2(\G\ba G)$ contains
 the unique irreducible infinite dimensional subrepresentation $V_\delta$ on which 
$\mathcal C$ acts by the scalar $\delta (1-\delta)$. 
Moreover $$V_\delta =\oplus_{\ell\in \z} \c \phi_\ell$$  where
 $\phi_\ell\in C^\infty(\G\ba G)\cap L^2(\G\ba G)$ is a unit vector (unique up to a unit) such that
$\phi_\ell(gk_\theta)=e^{2\ell i \theta} \phi_\ell(g)$ for all $g\in \G\ba G$ and $k_\theta\in K$ (cf. \cite{Bump}).

Let $\nu_i$ be a Patterson measure on $\partial(\bH^2)$
supported in $\Lambda(\G)$ with respect to $i\in \bH^2$ (\cite{Pa}, \cite{Su}).
Up to a scaling, $\nu_i$ is the weak-limit as $s\to \delta^+$
of the family of measures
$$\nu_{i}(s):=\frac{1}{\sum_{\gamma\in \G} e^{-sd(i, \gamma i)}}
\sum_{\gamma\in\G} e^{-sd(i, \gamma i)} \delta_{\gamma i}.$$

The $K$-invariant base eigenfunction $\phi_0\in L^2(\G\ba \bH^2)\cap C^\infty(\G\ba \bH^2)$ can
explicitly be given as the integral of the Poisson kernel against the Patterson measure \cite{Pa}:
\begin{equation}\label{vp1}
\phi_0(x+iy) = \int_{\br} \left(\frac{(u^2+1)y }{(x-u)^2 +y^2}\right)^\delta\; d\nu_i(u). \end{equation}
In the whole paper, we normalize $\nu_i$ so that $\|\phi_0\|_2=1$.

Let $\mathcal R$ and $\mathcal L$ be the raising and the lowering operators respectively given by
	$$\mathcal R = e^{2i\theta}\left(iy\tfrac{\partial}{\partial x} +y\tfrac{\partial}{\partial y} +\tfrac{1}{2i}\tfrac{\partial }{\partial \theta}\right)
\text{ and }\; \mathcal L = e^{-2i\theta}\left(-iy\tfrac{\partial }{\partial x}+y\tfrac{\partial }{\partial y} -\tfrac{1}{2i}\tfrac{\partial}{\partial\theta}\right).$$

For each $\ell \in \z_{\ge 0}$, set $$\psi_0^{(\ell)} := \mathcal R^{\ell}\phi_0,\quad\text{and}\quad
\psi_0^{(-\ell)} := \mathcal L^{\ell}\phi_0.$$

\begin{lem}\cite{BKS} \label{bo} For each $\ell\in \z_{\ge 0}$,
$$\|\psi_0^{(\pm \ell)}\|_2= \frac{\sqrt{\Gamma(\delta+\ell)\Gamma(1-\delta+\ell)}}{\sqrt{\Gamma (\delta) \Gamma(1-\delta)}}.
$$ where $\G(x)$ denotes the Gamma function for $x>0$.
\end{lem}

Since 
$\psi_0^{( \ell)}\in V_\delta$ and $\psi_0^{(\ell)}(gk_\theta)=e^{2\ell i \theta} \psi_0^{(\ell)}(g)$  by \cite[Sec. 2]{Bump},
 the unit vector $\phi_{\pm \ell}$ is now given as follows (up to a sign):
 $$\phi_{\pm \ell}=\frac{\psi_0^{(\pm \ell)}}{\|\psi_0^{ (\pm \ell)}\|_2 }.$$

Since $\phi_0>0$ and $\mathcal R$ is the complex conjugate of $\mathcal L$,
we have for each $\ell\in \z$,
$$
\phi_{- \ell}=\overline{\phi_{ \ell}}
.$$

\begin{thm}\label{l:rasing-lowering}
	For each  $\ell \in \z_{\ge 0}$, we have
$$\phi_{\pm \ell}(n_xa_yk_\theta) =e^{\pm 2\ell i\theta} \tfrac{\sqrt{\Gamma(\delta+\ell)\Gamma(1-\delta)}}{\sqrt{\Gamma (\delta) \Gamma(1-\delta+\ell)}}
\int_{\br} \left(\tfrac{(u^2+1)y}{(x-u)^2 +y^2}\right)^{\delta} \left(\tfrac{(x-u)\mp iy}{(x-u)\pm iy}\right)^{\ell}  d\nu_i(u).$$
\end{thm}
\begin{proof} 
By Lemma \ref{bo}, it suffices to show that
$$\psi_0^{(\pm \ell)}(n_xa_yk_\theta) =\tfrac{ e^{\pm 2\ell i\theta}\Gamma(\delta+\ell)}{\Gamma(\delta)}\int_{\br} \left(\tfrac{(u^2+1)y}{(x-u)^2 +y^2}\right)^{\delta} \left(\tfrac{(x-u)\mp iy}{(x-u)\pm iy}\right)^{\ell} \; d\nu_i(u).$$

Fix $u\in \R$. For $x\in\R$ and $y>0$, let
	$$f_u(x, y) := \frac{(u^2+1)y}{(x-u)^2+y^2}\;.$$

Then
	\begin{align*}
	\left(\mathcal R\phi_0\right)(n_xa_yk_\theta) =  \delta\cdot e^{2i\theta} \int_{\br} f_u(x,y)^\delta \cdot\frac{(x-u)-iy}{(x-u)+iy}\; d\nu_i(u)
	.\end{align*}
	
To use an induction, we assume that  
	\begin{align*}
	\left(\mathcal R^{\ell} \phi_0\right)(n_xa_yk_\theta) =  \tfrac{\Gamma(\delta+\ell)}{\Gamma(\delta)}\cdot e^{2i\ell\theta}\int_{\br} \left(\tfrac{(u^2+1)y}{(x-u)^2+y^2}\right)^\delta \left(\tfrac{(x-u)-iy}{(x-u)+iy}\right)^{\ell} \; d\nu_i(u)\;.
	\end{align*}

We compute
	\begin{multline*}
	\left(iy\frac{\partial}{\partial x} + y\frac{\partial}{\partial y}\right)\left(f_u^\delta(x, y)\cdot \left(\tfrac{(x-u)-iy}{(x-u)+iy}\right)^\ell\right)\\ =  
\delta\cdot f_u^\delta(x, y)\cdot \left(\tfrac{(x-u)-iy}{(x-u)+iy}\right)^{\ell+1}  + \ell\cdot f_u^\delta(x, y)\cdot \left(\tfrac{(x-u)-iy}{(x-u)+iy}\right)^{\ell}\cdot\left(\tfrac{-2iy}{(x-u)+iy}\right)
	\end{multline*}

and 
	\begin{multline*}
	\frac{1}{2i}\frac{\partial}{\partial\theta}\left(\mathcal R^\ell\phi_0\right)(n_xa_yk_\theta) \\
	= \tfrac{\Gamma(\delta+\ell)}{\Gamma(\delta)} \cdot \ell\cdot e^{2i\ell\theta}\int_\R\left(\tfrac{(u^2+1)y}{(x-u)^2+y^2}\right)^\delta\left(\tfrac{(x-u)-iy}{(x-u)+iy}\right)^\ell\; d\nu_i(u).
	\end{multline*}
Hence 
	\begin{align*}
	&\left(\mathcal R^{\ell+1}\phi_0\right)(n_xa_yk_\theta) = \left(\mathcal R\left(\mathcal R^{\ell} \phi_0\right)\right)(n_xa_yk_\theta)
	\\\\&= \tfrac{\Gamma(\delta+\ell+1)}{\Gamma(\delta)}\cdot e^{2i(\ell+1)\theta}\int_{\br}
f_u^\delta(x, y)  \left(\tfrac{(x-u)-iy}{(x-u)+iy}\right)^{\ell +1} \; d\nu_i(u)\;.
	\end{align*}
	The claim about the lowering operator follows from similar computations as above.
\end{proof}

\begin{cor} \label{ul}\label{ub}
For any $\ell\in \Z_{\ge 0}$, $x\in\br$ and $y>0$,
 $$\left|\phi_{\pm \ell}(n_xa_y)\right|  \ll \phi_0(n_xa_y)$$
with implied constant independent of $\ell$.
\end{cor}
\begin{proof} Since $\left|\frac{(x-u)-iy}{(x-u)+iy}\right| = \left|\frac{(x-u)+iy}{(x-u)-iy}\right|=1$,
 the claim follows from Theorem \ref{l:rasing-lowering}.
\end{proof}

\section{The average of $\phi_{\ell}$ over a closed horocycle}

We suppose that  
$\G\ba \G N$ is closed in $X$ and
 define \begin{equation}\label{xo}x_0=\begin{cases} \tfrac{1}{2}\min\{x>0: n_{x}\in N\cap \Gamma\}& \text{if $N\cap \Gamma\ne \{e\}$}\\
\infty&\text{ otherwise}\end{cases}\end{equation}
so that $2x_0$ is the period of the $N$-orbit $\G\ba \G N$.

For a function $\psi$ on $\G\ba G$ and $g\in G$, we define
	$$\psi^{N} (g) := \int_{n_x\in (N\cap \Gamma) \bsl N} \psi(n_x g)\;dx
=\int_{-x_0}^{x_0} \psi(n_x g) dx .$$

As
$\G\ba \Gamma N$ is a (stable) closed horocycle based at $\infty$, we have
either $\infty\notin \Lambda(\G)$ ($x_0=\infty$), or $\infty$ is a parabolic fixed point ($x_0<\infty$)
 \cite{Dalbo}. Since $\Lambda(\G)$ is a closed subset of boundary $\hat \br$,  $\infty\notin \Lambda(\G)$ implies that $\Lambda(\G)$ is
 a bounded subset of $\br$.

\begin{prop}\label{co} There exists $c_0>0$ such that
 $$\phi_0^N (a_y) = c_0\cdot  y^{1-\delta}.$$
\end{prop}
\begin{proof}
  When $N\cap \G$ is trivial, and 
hence $\Lambda(\G)$ is a bounded subset of $\br$,
we show by a direct computation (see \cite{K1}):
\begin{align*}\phi_0^N (a_y) &=\int_{u\in \Lambda(\G)} (u^2+1)^\delta d\nu_i(u)\cdot \int_{x\in \br}\left(\frac{y}{x^2+y^2}\right)^\delta dx
\\&= \omega_0 y^{1-\delta} \int_{t\in \br}\left(\frac{1}{1+t^2}\right)^\delta dt
\\ &=\omega_0  \frac{\sqrt{\pi} \Gamma(\delta-\tfrac 12)}{\Gamma(\delta)} y^{1-\delta}
\end{align*}
where $\omega_0=\int_{u\in \Lambda(\G)} (u^2+1)^\delta d\nu_i(u)$.

Now suppose $N\cap \G$ is non-trivial and hence $0<x_0<\infty$. Since $\phi_0^N(a_y)$ must satisfy the differential equation
$y^2\frac{\partial^2}{\partial y^2}\phi_0^N (a_y) =\delta(1-\delta)\phi_0^N(a_y) $,
there exist constants $c_0, d_0\in \br$ such that for all $y>0$
	$$\phi_0^N (a_y) = c_0 y^{1-\delta}+ d_0 y^{\delta}.$$
Since $\phi_0 >0$ and the above holds for all $y>0$, it follows that $c_0, d_0\ge 0$. 
We claim that $d_0=0$.

 Since $\G$ is finitely generated,
it follows that $\Gamma$ admits a fundamental domain $\mathcal F$ in $\mathbb H^2$ such that
$(-x_0, x_0) \times [Y_0,\infty)$ injects to $\mathcal F$ for some $Y_0\gg 1$.

Then \begin{align*} \|\phi_0\|^2 &\ge \int_{Y_0}^{\infty}\int_{-x_0}^{x_0}\phi_0(x+iy)^2  y^{-2} dx dy
\\  &\ge  \tfrac{1}{2x_0} \int_{Y_0}^{\infty} \left( \int_{-x_0}^{x_0}\phi_0(x+iy)dx\right)^2  y^{-2}  dy\\
&=   \tfrac{1}{2x_0} \int_{Y_0}^{\infty} \left(c_0y^{1-\delta} +d_0 y^{\delta} \right)^2  y^{-2}  dy\\ &\ge
\tfrac{d_0^2}{2x_0} \int_{Y_0}^{\infty}   y^{2\delta-2}  dy .
\end{align*}
Since $\delta>\tfrac 12$, $\|\phi_0\|=\infty$ unless $d_0\ne 0$.  Therefore $d_0=0$. Since $\phi_0>0$,
clearly $c_0>0$.
\end{proof}

\begin{lem}\label{zero} Let $\ell \in \z_{\ge 0}$. For any fixed $y>0$ and $0\le \theta<\pi$,
	$$\int_{(N\cap \Gamma)\bsl \Gamma} \frac{\partial}{\partial x} \phi_{\pm \ell}(n_x a_yk_\theta)\; dx = 0.$$
\end{lem}
\begin{proof}
 If $x_0<\infty$, then $$\int_{-x_0}^{x_0} \frac{\partial}{\partial x} \phi_{\pm \ell}(n_x a_yk_\theta)\; dx =
 \phi_{\pm \ell}(n_{x_0} a_yk_\theta) -\phi_{\pm \ell}(n_{-x_0} a_yk_\theta) .$$
Since  $\G n_{x_0}=\G n_{-x_0}$, the claim follows.

Suppose $x_0=\infty$. Since
 $$\phi_0(x+iy) = \int_{u\in \Lambda(\G)} \left(\frac{(u^2+1)y }{(x-u)^2 +y^2}\right)^\delta\; d\nu_i(u) $$
and $\Lambda(\G)$ is bounded, we have
$\phi_0(n_xa_y)\to 0$ as $|x|\to \infty$.
On the other hand,
\begin{align*} & \int_{-\infty}^{\infty} \frac{\partial}{\partial x} \phi_{\pm \ell}(n_x a_yk_\theta)\; dx \\ &=
\lim_{t\to \infty} \int_{-t}^t  \frac{\partial}{\partial x} \phi_{\pm \ell}(n_x a_yk_\theta)\; dx
\\ &= \lim_{t\to \infty} (  \phi_{\pm \ell}(n_{t} a_yk_\theta)  -
 \phi_{\pm \ell}(n_{-t} a_yk_\theta)).
\end{align*}
Since $|\phi_{\pm \ell}(n_{t} a_yk_\theta)|\ll \phi_0(n_{t} a_y)$ by Corollary \ref{ul}
and $\phi_0(n_ta_y)\to 0$ as $|t|\to \infty$,
the claim follows.
\end{proof}

\begin{thm}\label{l:abs_conv}
	For any $\ell \in \z_{\ge 0}$, 
	$$\int_{(N\cap \Gamma)\bsl N}\phi_{\pm \ell}(n_xa_y)\; dx = 
c_0 \frac{\sqrt{\Gamma(\delta)\Gamma(\ell +1-\delta)}}{\sqrt{\Gamma(1-\delta)
\Gamma(\delta+\ell)}} \;y^{1-\delta}.$$
In particular, for each $y>0$, 
$$\int_{(N\cap \Gamma)\bsl N}\phi_{\pm \ell}(n_xa_y)\; dx =O( y^{1-\delta})$$
with the implied constant independent of $\ell$.
\end{thm}
\begin{proof}
	By  Theorem \ref{ub},  
	\begin{align*}
	\left|\int_{(N\cap \Gamma)\bsl N} \phi_{\pm \ell} (n_xa_y)\; dx\right| &\leq \int_{(N\cap \Gamma)\bsl N} \left|\phi_{\pm \ell}(n_xa_yk_\theta)\right|\; dx
\\	&\ll  \int_{(N\cap \Gamma)\bsl N} \phi_0(n_xa_y)\; dx .
	\end{align*}
Hence by Proposition \ref{co}, the integral $\phi_{\pm \ell}^N (a_y)$ converges absolutely. 
To use an induction, we assume the following 
	\begin{equation} \label{ind1} \psi_0^{(\pm \ell)N} (a_yk_\theta)
 = e^{\pm2 \ell i\theta} c_0\frac{\Gamma(\ell+1-\delta)}{\Gamma(1-\delta)}y^{1-\delta} \end{equation}
	is true.  Then  applying Lemma \ref{zero},
	\begin{align*}&
	\psi_0^{(\pm (\ell+1))N}(a_yk_\theta) \\&= e^{\pm 2i\theta} \int_{(N\cap \Gamma)\bsl N} \left(\pm iy \frac{\partial }{\partial x}+y\frac{\partial }{\partial y}\pm \frac{1}{2i}\frac{\partial }{\partial \theta}\right)\psi_0^{(\pm \ell)}(n_x a_yk_\theta)\; dx\;
	\\&= e^{\pm 2i\theta} \cdot \left(y\frac{\partial }{\partial y}\pm \frac{1}{2i}\frac{\partial}{\partial \theta}\right)
\psi_0^{(\pm \ell)N}(a_yk_\theta)
	\\& =e^{\pm 2i\theta}\cdot
\left(y\frac{\partial }{\partial y}\pm \frac{1}{2i}\frac{\partial}{\partial \theta}\right)
\left( e^{\pm 2\ell i\theta} c_0\frac{\Gamma(\ell +1-\delta)}{\Gamma(1-\delta)}y^{1-\delta} \right)
	\\& = e^{\pm 2(\ell+1)i\theta}  \cdot 
\left(c_0\frac{ \Gamma(\ell +1-\delta)}{\Gamma (1-\delta)}\cdot \left((1-\delta) + \ell\right) y^{1-\delta} \right)
	\\& = e^{\pm 2(\ell+1)i\theta} \cdot \left(c_0\frac{\Gamma (\ell +1+(1-\delta))}{\Gamma (1-\delta)} y^{1-\delta}  \right) \end{align*}
since $z\Gamma(z) =\Gamma(z+1)$.  This proves \eqref{ind1} for all positive integer $\ell$.
Hence the claim follows from Lemma \ref{bo}.
\end{proof}

\section{Thickening of $\phi_{\ell}^N$}\label{w}
 The following lemma is proved in a greater generality in \cite{K1} for all $L^2$-eigenfunctions
in the discrete spectrum of $L^2(\G\ba \bH^2)$. 
\begin{lem}\label{beta}  Suppose that $\infty\notin \Lambda(\G)$
and let $J\subset \br $ an open subset containing $\Lambda(\G)$.
For all $0<y<1$, $$ \int_{J^c} \phi_0(n_x a_y)\; dx \ll y^{\delta}$$
with the implied constant independent of $y$.
\end{lem}

\begin{proof}
As we concern only the base eigenfunction $\phi_0$, this can be shown in a simpler way.
Let $\e_0:=\inf_{x\notin J, u\in \Lambda(\G)} |x-u|>0$.
Then by the change of variable $w=\frac{x-u}y$ 
 we have

\begin{align*} \int_{J^c} \phi_0(n_x a_y)\; dx &\le  2 y^{1-\delta}\int_{u\in \Lambda(\G)}(u^2+1)^\delta d\nu_i(u)
\cdot \int_{w=\e_0/ y}^\infty \left( \frac{1}{w^2+1}\right)^\delta dw\end{align*}

The latter integral can be evaluated explicitly as an incomplete Beta function which has known asymptotics:
$$
\int_{\e_0/y}^{\infty} \left( {1 \over w^2+1}\right)^{\delta} dw = c\ \beta_{y^2 / \e_0^2}( \delta-1/2, 1-\delta),
$$
where 
$$
\beta_z(\alpha,\beta)\ll z^{\alpha}.
$$

Therefore
$$ \int_{J^c} \phi_0(n_x a_y)\; dx \ll y^{1-\delta} y^{2(\delta-1/2)}=y^\delta .$$
\end{proof}

\begin{lem}\label{es} Suppose that $\infty\notin \Lambda(\G)$
and let $J\subset \br $ an open subset containing $\Lambda(\G)$. 
Fix $\ell \in \z_{\ge 0}$.   For all $0<y <1$,
	$$\phi_{\pm \ell}^N(a_y)  = \int_J \phi_{\pm \ell}(n_x a_y)\; dx +
O(y^{\delta})$$
with the implied constant independent of $\ell$ and $y$.
\end{lem}
\begin{proof}
Note that, by Corollary \ref{ul},
	$$\left|\int_{x\in J^c} \phi_{\pm \ell}(n_xa_y)\; dx\right|
	 \leq \int_{x\in J^c} \left|\phi_{\pm \ell}(n_x a_y)\right|\; dx
	 \ll  \int_{J^c} \phi_0(n_x a_y)\; dx $$
since $\phi_0$ is a positive function. 

Hence the claim follows from Lemma \ref{beta}.
\end{proof}

Setting $N^-:=\left\{\bpm 1 & 0 \\ x & 1\ebpm : x\in \R\right\}$,
the product map $N\times A\times N^-\to G$ is a diffeomorphism at a neighborhood of $e$.

Let $dk$ be the invariant probability measure on $K$ and
denote by $dg$ the Haar measure on $G$: $dg= \frac{1}{y^2} dxdydk $ for $g=n_xa_yk$.
 Let $\nu$ be a smooth measure on $AN^-$ such that $dn_x\otimes d\nu(a_yn_{x'}^-) = dg$.
When $\infty\notin \Lambda(\G)$, fix a bounded open interval $J$ which contains
$\Lambda(\G)$ and choose a compactly supported smooth function $0\le \eta\le 1$ on $N$
with $\eta|_{J}=1$. Otherwise, let $\eta=1$ on $[-x_0, x_0)$, which is a fundamental domain for
$N\cap \G \ba N$.
We denote by $U_\e$ the $\e$-neighborhood of $e$ in $G$.
Fix $\e_0>0$ so that
the multiplication map
 $${\rm supp}(\eta)\times (U_{\epsilon_0} \cap AN^-) \to {\rm supp}(\eta) \left(U_{\epsilon_0}\cap AN^- \right) \subset \Gamma\bsl G$$
	is a bijection onto its image. 
	For each $0<\epsilon< \epsilon_0$, let $r_\epsilon$ be a non-negative smooth function in $AN^-$ whose support is contained in 
		$W_\epsilon:= (U_\epsilon\cap A) (U_{\epsilon_0} \cap N^-) $ and 
		$\int_{W_\epsilon}r_\epsilon\; d\nu=1\;.$
		
	We define the following function $\rho_{\eta, \epsilon}$ on $\Gamma\bsl G$:
\begin{equation*}\rho_{\eta, \epsilon}(g)= \begin{cases} \eta(n_x)\cdot r_\epsilon(a_yn_{x'}^-)
&\text{ for $g=n_xa_yn_{x'}^-\in {\rm supp}(\eta)W_\e$ }\\ 0 &\text{ for $g\notin
{\rm supp}(\eta)W_\e$}. \end{cases}\end{equation*}

The inner product on $L^2(\G\ba G)$ is given
by $$\la \psi_1,\psi_2\ra=\int_{\G\ba G} \psi_1(g)\overline{\psi_2(g)} \; dg.$$

 Set $c_\ell= c_0\frac{\sqrt{\Gamma(\delta)\Gamma(\ell +1-\delta)}}{\sqrt{\Gamma(1-\delta)
\Gamma(\delta+\ell)}}$ 
so that for each $\ell\in \z_{\ge 0}$,
$$\phi_{\pm \ell}^N(a_y)= c_\ell y^{1-\delta}$$ by Theorem \ref{l:abs_conv}.

\begin{prop}\label{l:approx_matrix_coeff}
For any $\ell\in \z_{\ge 0}$, we have for all positive $\e \ll 1$,
	$$
\left<a_y\phi_{\pm \ell}, \rho_{\eta, \epsilon}\right> =c_\ell y^{1-\delta}+
O_\eta(\ell \epsilon y^{1-\delta})+O_\eta(\ell y^{\delta})  $$
with the implied constants independent of $\ell$ and $y$.
 \end{prop}
\begin{proof} 
For $x\in \br$, set $\theta_x:=-\arctan x$.
Then we compute $$n_x^-=n_{-\sin\theta_x} a_{\cos^2\theta_x} \kappa_{\theta_x} .$$
For $h=a_{y_0}n_{x_0}^-\in W_{\epsilon}$, $n_x\in N$ and $y>0$,
we can write $$n_x h a_y= n_x n_{-y_0 y\sin\theta_{x_0 y}} a_{y_0 y\cos^2\theta_{x_0 y}}\kappa_{\theta_{x_0 y}} .$$

Since $y_0=1+ O(\epsilon)$ and $x_0=O(1)$,
 we have \begin{enumerate}
          \item $\cos^2\theta_{x_0y}=1+O(\epsilon)$;
\item  $y_0 \cos^2\theta_{x_0 y} = 1+O(\e)$
\item  $y_0 y\sin\theta_{x_0 y}=O(\e)$
\item $e^{2\ell i \theta_{x_0y}}= 1+O(\ell \e) $
\end{enumerate}
 where the implied constants are independent of $0<y<1$.

For $y_1:=	y_0 y\cos^2\theta_{x_0 y}=y(1+O(\e))$ and $x_1:=-y_0 y\sin\theta_{x_0 y}=O(y)$,
$$\phi_{\ell}(n_x h a_y)= e^{2\ell i \theta_{x_0y}} \phi_{\ell}(n_{x+x_1}  a_{y_1})
=(1+O(\ell \e))\cdot \phi_{\ell}(n_{x+x_1}  a_{y_1})$$
and hence
$$\int_{(N\cap \Gamma)\bsl N} \phi_{\ell}(n_x h a_y)\cdot\eta(n_x)\; dx
	=  (1+O(\ell \e)) \int_{(N\cap \Gamma)\bsl N} \phi_{\ell}(n_xa_{y_1})(\eta(n_x) +O (y))\; dx 
$$
where the implied constants are independent of $\ell$. 


Hence  using Lemma \ref{es} and $c_\ell=O(1)$, we deduce
	\begin{align*}	&\int_{(N\cap \Gamma)\bsl N} \phi_{\ell}(n_x h a_y)\cdot\eta(n_x)\; dx
	\\&=   (1+O(\ell \e))  \int_{(N\cap \Gamma)\bsl N} \phi_{\ell}(n_xa_{y_1})(\eta(n_x) +O(y))\; dx 
	 \\&= \int_{(N\cap \Gamma)\bsl N} \phi_{\ell}(n_xa_{y_1})\eta(n_x) dx +O(\ell\e \phi_{\ell}^N(a_{y_1}) ) 
 \\&= c_\ell y_1^{1-\delta} + O( \ell \e y_1^{1-\delta}) + O( \ell y_1^{\delta})
\\&= c_\ell y^{1-\delta} + O( \ell \e y^{1-\delta}) + O( \ell y^{\delta}). 
\end{align*}

Since $\int r_\e d\nu(h)=1$,
it follows that  \begin{align*}\left<a_y\phi_{\ell}, \rho_{\eta, \epsilon}\right>&=\int_{W_\e}r_\e(h)\int_{\G\cap N\ba N} 
                  \phi_{\ell}(n_xha_y) \eta(n_x)\; dx\; d\nu(h) \\&=c_\ell y^{1-\delta} + O(\ell \e y^{1-\delta}) + O(\ell y^{\delta}) .
                 \end{align*}
\end{proof}


\section{Equidistribution of a closed horocycle}

For $\psi\in C_c^\infty(\Gamma\bsl G)$, our goal is to compute
	$$\psi^N(a_y): = \int_{(N\cap \Gamma)\bsl N} \psi(n_x a_y)\; dx$$
in terms of $\phi_{\ell}^N(a_y)$ for $\ell \in \Z$.

Let $\{Z_1,Z_2, Z_3 \}$ be a basis of the Lie algebra of $G$.
For $\psi \in C^\infty (\G\ba G)\cap L^2(\G\ba G)$, and $m\ge 1$,  we consider the following
 Sobolev norm $\mathcal S_m(\psi )$:
$$  \mathcal S_m(\psi )=\max \{\| Z_{i_1}\cdots Z_{i_n} (\psi) \|_2 :
1\le i_j\le 3, \;\; 0\le n\le m \}.$$

\begin{lem}\label{haha} Fix $m\in \N$.
 For any $\psi\in C^\infty(\Gamma\bsl G)\cap L^2(\G\ba G)$ and for all $|\ell|$ large,
$$ \left| \left<\psi, \phi_{\ell}\right> \right|\ll (|\ell| +1)^{-m}  \mathcal S_m(\psi).$$
In particular, $$\sum_{\ell \in \Z}  |c_\ell \left<\psi, \phi_{\ell}\right>|<\infty .$$ \end{lem}
\begin{proof} The element $H=\begin{pmatrix} 0 & 1 \\ -1 & 0
                       \end{pmatrix}$ in the Lie algebra of $G$ corresponds to the differential operator
$\frac{\partial}{\partial \theta}$
and  $\frac{\partial}{\partial \theta} \phi_{\ell} =2\ell i \phi_{\ell} .$
Hence
$$\left< \frac{\partial}{\partial \theta}\psi , \phi_{\ell} \right> = \left< \psi, -\frac{\partial}{\partial \theta} \phi_{\ell} \right>
 = 2\ell i\left< \psi ,\phi_{\ell} \right> .$$
Similarly, $$\left| \left< \frac{\partial^m}{\partial \theta^m}\psi , \phi_{\ell} \right>\right| 
 = 2^m \ell^m |\left< \psi ,\phi_{\ell} \right> |.$$
Hence $$\left|\left< \psi ,\phi_{\ell} \right>\right| \ll \frac{1}{(| \ell|+1) ^m} \cdot \left\| \frac{\partial^m}{\partial \theta^m}\psi\right\| _2$$
proving the first claim. 
Since $c_\ell=O(1)$, the second claim follows. \end{proof}

Fix $ 1/2<s_1<\delta$ so that there is no eigenvalue of the Laplacian
between $s_1(1-s_1)$ and $\delta (1-\delta)$ in $L^2(\G\ba \bH^2)$. 
\begin{lem}\label{l:matrix_coeff}
For any $\psi_1,\psi_2\in C^\infty(\Gamma\bsl G)$ with  $\mathcal S_1(\psi_i)<\infty$, and $0< y< 1$, we have
	$$\left<a_y\psi_1, \psi_2\right> = \sum_{\ell \in \Z} \left< \psi_1, \phi_{\ell}\right>\left<a_y\phi_{\ell}, \psi_2\right> + 
O\left(y^{1-s_1} \cdot \mathcal S_1(\psi_1)\cdot \mathcal S_1(\psi_2)\right).$$
\end{lem}	
\begin{proof} 
	We have	$L^2(\Gamma\bsl G) = V_{\delta}\oplus V_\delta^{\perp}$
	where $V_\delta^\perp$ does not contain any complementary series $V_s$ with parameter $s>\delta$. We can write 
	$$\psi_1 = \sum_{\ell \in \Z} \left<\psi_1, \phi_{\ell}\right>\phi_{\ell} + \psi_1^{\perp}$$
with $\psi_1^{\perp}\in V_\delta^\perp$ since 
	$\left<\psi_1-\sum_{\ell \in \Z}\left<\psi_1, \phi_{\ell}\right>\phi_{\ell}, \phi\right>=0$
	for any $\phi\in V_{\delta}$. Hence
	\begin{align*}
	\left<a_y\psi_1, \psi_2\right> & = \sum_{\ell \in \Z} \left<\psi_1, \phi_{\ell}\right>\left<a_y\phi_{\ell}, \psi_2\right> +\left<a_y\psi_1^{\perp}, \psi_2\right> .
	\end{align*}

On the other hand, by the assumption on $s_1$, we have (cf. the proof of corollary 5.6 in \cite{KO2})
$$\left< a_y \psi_1^\perp ,\psi_2\right> \ll  y^{1-s_1}\cdot \mathcal S_1(\psi_1)\cdot \mathcal S_1(\psi_2).$$ This implies the claim.
\end{proof}
	


We refer to \cite{KO1} for the next lemma: 
\begin{lem}\label{l:hat_psi}
For $\psi\in C_c^\infty(\Gamma\bsl G)$, there exists $\widehat \psi\in C_c^\infty(\Gamma\bsl G)$ such that
	\begin{enumerate}
	\item for all small $\epsilon>0$, and $h\in U_\epsilon$, 
		$$\left|\psi(g)-\psi(gh)\right|\leq \epsilon\cdot\widehat \psi(g)$$
		for all $g\in \Gamma\bsl G$. 
	\item For all $m\in \N$,
$\mathcal S_m(\widehat \psi)\ll \mathcal S_{3}(\psi)$  where the implied constant depends only on ${\rm supp}(\psi)$.
	\end{enumerate}
\end{lem}

\begin{lem}\label{jjj} Let $\infty\notin \Lambda(\G)$.
For a fixed compact subset $Q$ of $ G$, there exists a bounded subset $J\subset \br$
such that $n_xa_y\notin \Gamma Q$ for all $x\notin J$ and any $0<y<1$.
\end{lem}
\begin{proof}
If not, there exist sequences $x_j\to \infty$, $y_j\in \br$, $\gamma_j\in \G$ and $w_j\in Q$
such that $n_{x_j}a_{y_j}=\gamma_j w_j$. As $Q$ is compact, we may assume $w_j \to w\in Q$.
On the other hand, 
$n_{x_j}a_{y_j}(i)=x_j+y_j i\to \infty$ as $x_j \to \infty$.
Hence $\gamma_j(w)\to \infty$, implying that $\infty\in \Lambda(\G)$, contradiction.
\end{proof}

\begin{thm}\label{m1} For any $\psi\in C_c^{\infty}(\Gamma\bsl G)$
$$\psi^N(a_y) =\sum_{\ell \in \Z}  c_\ell \left<\psi, \phi_{\ell}\right> y^{1-\delta}
+O(\mathcal S_3(\psi) y^{  1-\delta+\tfrac{2{\s}_\G}5}) .$$
\end{thm}
\begin{proof}  
By Lemma \ref{jjj}, there exists a bounded open subset $J$ such that
$\psi(n_xa_y)=0$ for all $x\notin J$ and all $0<y<1$.  
When $\infty\notin \Lambda(\G)$, we will assume that $J$ contains $\Lambda(\G)$, by enlarging $J$ if necessary,
and otherwise $J=(-x_0, x_0)$.

Choose a non-negative function $\eta\in C_c^\infty (N\cap \G \ba N)$ such that $\eta|_J=1$. Then
	$$I_\eta(\psi)(a_y): = \int_{(N\cap \Gamma)\bsl N} \psi(n_xa_y) \eta(n_x)\; dx =
  \psi^N(a_y)\;.$$

Let $\e_0, W_\e, r_\e,$ and $ \rho_{\eta, \e}$ be as defined in section \ref{w}
with respect to this $J$ and $\eta$. 
Since $r_\e$ is the approximation of the identity in the $A$ direction, $\mathcal S_1(\rho_{\eta, \e})=O_\eta(\e^{-3/2})$.
 For any $0< y< 1$, and any small $\epsilon>0$, we have (see the proof of \cite[Prop. 6.6]{KO2})
	\begin{equation}\label{l:aprox_matrix_coeff} \left|I_\eta(\psi)(a_y)-\left<a_y\psi, \rho_{\eta, \epsilon}\right>\right| \ll (\epsilon+y)\cdot I_\eta(\widehat \psi)(a_y).\end{equation}

Fix $1/2<s_1<\delta$ as in Lemma \ref{l:matrix_coeff}. Let $k$ be an integer bigger than $\frac{5(1-\delta)}{2(\delta-s_1)}+1 $. 
Setting $\psi_0(g) := \psi(g)$, we define for $1\leq j\leq k$, inductively
	$$\psi_j(g) :=\widehat\psi_{j-1}(g)$$
where $\widehat\psi_{j-1}$ is given by Lemma \ref{l:hat_psi}. 
Applying Lemma \ref{l:aprox_matrix_coeff} to each $\psi_j$, we obtain for $0\leq j\leq k-1$, 
	\begin{align*}
	I_{\eta}(\psi_j)(a_y) &= \left<a_y\psi_j, \rho_{\eta, \epsilon}\right> + 
O\left((\epsilon+y)\cdot I_\eta(\widehat\psi_j)(a_y)\right)
	\\&=\left<a_y\psi_j, \rho_{\eta, \epsilon}\right> + O\left((\epsilon+y)\cdot I_\eta(\psi_{j+1})(a_y)\right)
	\end{align*} and
	$$I_\eta(\psi_j)(a_y) = \left<a_y \psi_j, \rho_{\eta, \epsilon}\right> +
O_\eta\left((\epsilon+y)\mathcal S_{1}(\psi_k)\right).$$
 
Note $$|\la \psi_j, \phi_{\ell}\ra |=(|\ell|+1)^{-3}O(\mathcal S_3(\psi))$$ by Lemmas \ref{haha} and \ref{l:hat_psi}.

Since $\left<a_y\phi_{\ell}, \rho_{\eta, \epsilon}\right> =O(\ell y^{1-\delta})$ by Proposition \ref{l:approx_matrix_coeff}, we deduce
$$\sum_{\ell \in \Z} |\left<\psi_j, \phi_{\ell}\right>\left<a_y\phi_{\ell}, \rho_{\eta, \epsilon}\right> |
= \sum_{\ell \in \Z}(|\ell|+1)^{-2}   y^{1-\delta} O(\mathcal S_3(\psi)) 
=   y^{1-\delta} O(\mathcal S_3(\psi)) .$$
Hence
by Lemma \ref{l:matrix_coeff}, we deduce that for each $1\leq j \leq k-1$, 
	\begin{align*}
	\left<a_y\psi_j, \rho_{\eta, \epsilon}\right> &=
\sum_{\ell \in \Z} \left<\psi_j, \phi_{\ell}\right>\left<a_y\phi_{\ell}, \rho_{\eta, \epsilon}\right> +
 O\left(y^{1-s_1}\cdot\mathcal S_1(\psi_j)\cdot\mathcal S_1(\rho_{\eta, \epsilon})\right)
	\\&= O (\mathcal S_3(\psi) \cdot y^{1-\delta} ) +O(y^{1-s_1}\cdot\mathcal S_1(\psi_j)\cdot\mathcal S_1(\rho_{\eta, \epsilon}) )
	\\&=\mathcal S_3(\psi)  \cdot O( y^{1-\delta} +\epsilon^{-3/2}y^{1-s_1}).
	\end{align*}

Hence for any $0<y<\epsilon$, using Proposition \ref{l:approx_matrix_coeff}, we deduce
	\begin{align*}
	&I_\eta(\psi)(a_y) = \left<a_y\psi, \rho_{\eta, \epsilon}\right> +
\sum_{j=1}^{k-1}O\left(\left<a_y\psi_j, \rho_{\eta, \epsilon}\right>(\epsilon+y)^j\right) +
 O_\psi ((\epsilon+y)^k)
\\&=\left<a_y\psi, \rho_{\eta, \epsilon}\right> 
 + O(\epsilon\cdot y^{1-\delta} + \epsilon^{-3/2} y^{1-s_1} +\epsilon^k)
	\\&=\sum_{\ell \in \Z} \left<\psi, \phi_{\ell}\right>\left<a_y\phi_{\ell},
 \rho_{\eta, \epsilon}\right> +
 O(\mathcal S_3(\psi)  (\epsilon\cdot y^{1-\delta} + \epsilon^{-3/2} y^{1-s_1} +\epsilon^k))
 \\&= \sum_{\ell \in \Z} \left<\psi, \phi_{\ell}\right> c_\ell y^{1-\delta} +
\mathcal S_3(\psi) O(y^\delta+\epsilon\cdot y^{1-\delta} + \epsilon^{-3/2} y^{1-s_1} +\epsilon^k).
\end{align*}

By equating $\epsilon\cdot y^{1-\delta} $ and $ \epsilon^{-3/2} y^{1-s_1}$
we put $\e=y^{2(\delta-s_1)/5}$ and obtain
$$I_\eta(\psi)(a_y) =\sum_{\ell \in \Z} c_\ell \left<\psi, \phi_{\ell}\right>  y^{1-\delta}
+\mathcal S_3(\psi) O(y^{1-\delta +\frac{2(\delta-s_1)}{5}}).$$
\end{proof}

\begin{rmk}
 \rm
Suppose that $\psi\in C_c^\infty(\G\ba G)$ is a real valued function. Since $c_\ell =c_{-\ell}$
and $\phi_{-\ell}=\overline{\phi_{\ell}}$ for each $\ell \in \z$,
we have
$$\sum_{\ell\in \z} c_{\ell}\la \psi, \phi_\ell\ra\in \br $$
as expected.
 \end{rmk}

\section{Comparison of main terms and Burger-Roblin measure
as a distribution}
Recall the Patterson measure $\nu_i=\nu_i^\G$ on the boundary and $\phi_0=\phi_0^\G$ given by
$$\phi_0(x+iy) = \int_{\br} \left(\frac{(u^2+1)y }{(x-u)^2 +y^2}\right)^\delta d\nu_i(u)$$ from section 1.
 Note that $$\phi_0^\G(e)=|\nu_i^\G|.$$
As before, we normalize $\nu_i$ so that $\|\phi_0\|_2=1$. 

For $\xi\in \partial(\bH^2)$ and $z_1, z_2\in \bH^2$,
recall the Busemann function:
$$\beta_\xi(z_1,z_2)=\lim_{s\to \infty}d(z_1,\xi_s)-d(z_2,\xi_s)$$
where $\xi_s$ is a geodesic ray tending to $\xi$ as $s\to \infty$.

Using the identification of $\T^1(\bH^2)$ and $G$,
we give the definition of 
the Bowen-Margulis-Sullivan measure $m^{\BMS}$  on $\Gamma\ba G$.
For $u\in \T^1(\bH^2)$, we denote by $u^+$ and $u^-$ the forward and the backward endpoints
of the geodesic determined by $u$, respectively.
The correspondence $$u\mapsto (u^+, u^-, t:=\beta_{u^-}(i, \pi(u)))$$ gives a homeomorphism between
the space $\T^1(\bH^2)$ with  
 $(\partial(\bH^2)\times \partial(\bH^2) - \{(\xi,\xi):\xi\in \partial(\bH^2)\})  \times \br $
where $\pi: G\to G/K=\bH^2$ is the canonical projection.
Define the measure $\tilde m^{\BMS}$ on $G$:
$$d \tilde m^{\BMS}(u)=e^{\delta \beta_{u^+}(i, \pi(u))}\;
 e^{\delta \beta_{u^-}(i,\pi(u)) }\;
d\nu_i(u^+) d\nu_i(u^-) dt .$$
This measure is left $\G$-invariant and hence induces a measure $m^{\BMS}$ on $\G\ba G$.
As $\G$ is finitely generated, we have $|m^{\BMS}|<\infty$.

Roblin obtained the following in his thesis \cite{RoT}:
\begin{thm}[Roblin] \label{rot} For $\delta >1/2$,
 $$\|\phi_0\|_2^2=|m^{\BMS}|\cdot\int_{\br} \frac{dx}{(1+x^2)^\delta}$$
\end{thm}

As we have normalized $\nu_i$ so that $\|\phi_0\|_2=1$ and $\phi_0(i)=|\nu_i|$, we deduce
$\tfrac{1}{|m^{\BMS}|}=\int_{\br} \frac{dx}{(1+x^2)^\delta}$.
As $\delta>\tfrac 12$, we have $$\tfrac{\sqrt \pi \Gamma(\delta-\tfrac 12)}{\Gamma(\delta)}=
\int_{-\infty}^\infty  \tfrac{{dx}}{(1+x^2)^{\delta}}.$$

To describe the equidistribution result of $N\cap \G\ba N a_y$ from \cite{Ro}, we recall
 $m^{\BR}_N$ and $\mu^{\PS}_N$.
First, define the measure $\tilde m^{\BR}_{N}$ on $ G$ as follows :
for $\psi\in C_c( G)$,
$$d \tilde m^{\BR}_N(u)=e^{\delta \beta_{u^+}(i, \pi(u))}\;
 e^{ \beta_{u^-}(i,\pi(u)) }\;
d\nu_i(u^+) dm_i(u^-) dt $$
where $m_i$ is the $K$-invariant probability measure on $\partial(\bH^2)$.

This measure is left $\G$-invariant and right $N$-invariant,
and the Burger-Roblin measure $m^{\BR}_{N}$ (associated to the stable horocyclic subgroup $N$) 
is the measure on $X$ induced from $\tilde m^{\BR}_N$.

Consider the measure $\mu^{\PS}_N$ on $N$ given by
 $$d\mu^{\PS}_N(n_x)=e^{-\delta\beta_{x} (i, x+i)} d\nu_i(x)=(1+x^2)^{\delta} d\nu_i(x).$$ 
This induces a measure on $(N\cap \G)\ba N$ for which we use the same notation. 
Since $\mu^{\PS}_N$ is supported in $(N\cap \G)\ba (\Lambda(\G)-\{\infty\})$, we have
 $ \mu^{\PS}_N ((N\cap \Gamma) \ba N)<\infty$.

\begin{thm} (\cite{Ro}, see also \cite{OS})\label{ros} Let $\delta>0$ and $(N\cap \G)\ba N$ be closed.
 For any $\psi\in C_c(\Gamma\ba G)$, 
$$\lim_{y\to 0} y^{\delta-1} \psi^N(a_y) = \frac{\mu^{\PS}_N (N\cap \Gamma \ba N)}{|m^{\BMS}|} m^{\BR}_N(\psi)  .$$
\end{thm}

Comparing the main terms of Theorem \ref{ros} and Theorem \ref{m1} and using Theorem \ref{rot}, we deduce
the following interesting identity of the Burger-Roblin measure considered
as a distribution on $\G\ba G$: 
\begin{thm}\label{brd} Let $\delta>1/2$ and $(N\cap \G)\ba N$ be closed. For any $\psi\in C_c^\infty(\Gamma\ba G)$,
 $$\kappa_{\G} \cdot m^{\BR}_N(\psi) =\sum_{\ell \in \Z}  c_\ell\left<\psi, \phi_{\ell}\right>$$
where $\kappa_{\G}=\tfrac{\sqrt \pi \Gamma(\delta-\tfrac 12)}{\Gamma(\delta)}
 \int_{(N\cap \G)\ba N} {(1+x^2)^{\delta}} {{d\nu_i(x)}}.$
\end{thm}

Hence Theorem \ref{main3} follows from
Theorems \ref{main2} and \ref{brd}.

\section{Application to counting in sectors}
\subsection{}
Let $Q$ be a ternary indefinite quadratic form  and $v_0\in \br^3$ be any vector
such that $Q(v_0)=0$. 
 Let $\G_0 < \SO_Q^\circ(\br)$ be a finitely generated discrete subgroup
with $\delta>1/2$. Suppose that $v_0\G_0$ is discrete.

 Let $\|\cdot \|$ be {\it any} norm in $\br^3$. We can find a representation $\iota: G=\PSL_2(\br)\to \SO_Q(\br)$ so that
the stabilizer of $v_0$ in $\PSL_2(\br)$ via $\iota$ is
the upper triangular subgroup $N$. Let $\Omega\subset K$ be a Borel subset
 and consider the sector
$$S_T(\Omega):=\{ v\in v_0. A\Omega :\|v\|<T\} $$
given by $\Omega$.

\begin{thm}\label{ec}
Suppose that
$\Omega$ has only finitely many connected components.
 Then $$\#\{v\in v_0\Gamma_0\cap S_T(\Omega)\}=
\frac{\kappa_{\iota^{-1}(\G_0)}}{\delta} 
\left(
\int_{k^{-1}\in \Omega} \frac{ d\nu_i(k(0))} {\|v_0k^{-1} \|^{\delta}}  \;
 \right) T^\delta +O(T^{\delta -\tfrac{4 {\s}_\G}{55}})) .$$
\end{thm}

Theorem \ref{ec} was obtained in \cite{OS} but without an error term.
Deducing this theorem from Theorem \ref{main2} is a verbatim repetition of
the arguments in the section 8 of \cite{LO}. For the sake of completeness, we give a brief sketch here.


Let $\Gamma:=\iota^{-1}(\Gamma_0)$, and
let $U_\e$ be an $\e$-neighborhood of $e$ in $G$.

By the assumption on $\Omega$, the boundary of
 $\Omega^{-1}(0)$ consists of finitely many points and hence $\nu_i(\partial(\Omega^{-1}(0)))=0$, as $\nu_i$ is atom-free.
Moreover for all sufficiently small $\e>0$,
 there exists an $\e$-neighborhood $K_\e$ of $e$ in $K$
  such that for $\Omega_{\e+}=\Omega K_\e$ and 
  $\Omega_{\e-}=\cap_{k\in K_\e} \Omega k$,
 \begin{equation}\label{eq:600}
 \nu_i(\Omega_{\e +}^{-1} (0)-  \Omega_{\e-}^{-1}(0))\le \e.
 \end{equation}

By the strong wave front lemma \cite[Theorem 4.1]{GOS},  there exists $0<\ell_{0}<1$ such that for $T\gg 1$,
$$S_T(\Omega) U_{\ell_0 \e} \subset S_{(1+\e)T}(\Omega_{\e+}) \quad\text{and}\quad
S_{(1-\e)T}(\Omega_{\e-})\subset\cap_{u\in U_{\ell_0 \e}} S_T(\Omega) u .$$

Let
 $\psi_\e \in C_c^\infty(G)$ be a non-negative function supported in $U_{\ell_0 \e}$ with integral one, and
set $\Psi_\e(g)=\sum_{\gamma\in \G} \psi_\e(\gamma g)$.

Define the counting function $F_T^{\Omega}$ on $\G\ba G$ by
$$F_T^{\Omega}(g)=\sum_{\gamma\in N\cap \Gamma\ba \Gamma} \chi_{S_T(\Omega)}(v_0 \gamma g).$$

Then 
\[  
\la F_{(1-\e)T}^{\Omega_{\e-}} ,\Psi_\e\ra \le F_{T}^{\Omega}(e)
\le\la  F_{(1+\e)T}^{\Omega_{\e+}},\Psi_\e\ra.
\]

Let $\Psi_\e^{k}(g):=\Psi_\e(gk)$. 

The following  is a special case of \cite[Prop. 6.2]{OS} (see also \cite[Sec. 7]{KO2}):
\begin{prop} \label{special} Suppose that
$\nu_i(\partial(\Omega^{-1}(0)))=0$.
 For all small $\e>0$,
$$ \int_{k \in \Omega}
 \frac{  m^{\BR}_N(\Psi_\e^{k})}{\|v_0k\|^\delta}
 dk =\int_{k\in \Omega^{-1}} \frac{ d \nu_i(k(0))}{\|v_0k^{-1}\|^\delta} (1+O(\e)). $$
\end{prop}

\begin{prop}\label{ftt} For all $T\gg 1$,
 $$\la F_T^\Omega, \Psi_\e\ra_{L^2(\G\ba G)}=\frac{\kappa_\G \cdot T^\delta}{\delta}\int_{k\in \Omega^{-1}} \frac{ d \nu_i(k(0))}{\|v_0k^{-1}\|^\delta}  + O(\e T^\delta+ \e^{-9/2} T^{\delta -2{\s}_\G/5}))  .
$$
\end{prop}
\begin{proof} We only sketch a proof here and refer to \cite{LO} for details.
\begin{align*}\la F_T^{\Omega}, \Psi_\e\ra_{L^2(\G\ba G)}
&=  \int_{k\in \Omega} \int_{y>\tfrac{\|v_0k\|}{T}}
 \left(\int_{N\cap \Gamma\ba N} \Psi_\e (n_x a_y k) dx \right) y^{-2} dy dk
.\end{align*}

By Theorem \ref{main2} and Corollary \ref{brd}, 
\begin{align*}\int_{N\cap \Gamma\ba N} \Psi_\e (n_x a_y k) dx = \kappa_\G\cdot  m^{\BR}_N(\Psi_\e^{k}) \cdot y^{1-\delta}
+O(\e^{-9/2} y^{(1-\delta)+ 2{\s}_\G/5} )\end{align*}
as $\mathcal S_3(\Psi_\e)=\e^{-9/2}$.
We then  deduce, using Proposition \ref{special},
\begin{align*}&\la F_T^\Omega, \Psi_\e\ra_{L^2(\G\ba G)}=
\frac{\kappa_\G \cdot T^\delta}{\delta} \int_{k\in \Omega^{-1}} \frac{ d \nu_i(k(0))}{\|v_0k^{-1}\|^\delta} (1+O(\e))
  +O(\e^{-9/2}T^{\delta-2{\s}_\G/5})  .
\end{align*} 
\end{proof}

By Proposition \ref{ftt},
\begin{align*} 
\la F_{(1+\e)T}^{\Omega_{\e\pm}}, \Psi_\e \ra 
&= \frac{\kappa_\G \cdot T^{\delta}}{\delta} \int_{k\in \Omega^{-1}_{\e \pm}} 
\frac{ d \nu_i(k(0))}{\|v_0k^{-1}\|^\delta} 
  +O(\e T^\delta + \e^{-9/2} T^{\delta-2{\s}_\G/5}).
\end{align*}

Therefore by solving $\e^{-9/2}T^{-2{\s}_\G/5}=\e$ for $\e$, we finish the proof of Theorem \ref{ec}.

\subsection{}
Let $Q$ be a ternary indefinite quadratic form over $\Q$  and $v_0\in \Q^3$ be any vector
such that $Q(v_0)=0$. 
 Let $\G< \SO_Q(\z)$ be a finitely generated subgroup
with $\delta>1/2$. 

 For a square-free integer $d$, 
consider the subgroup of $\G$ which stabilizes $v_0$ mod $d$:
$$\G_d:=\{\gamma\in \G: v_0 \gamma \equiv v_0 \mod d\}.$$
Clearly it satisfies $\text{Stab}_{\G} {v_0} =\text{Stab}_{\G_d} {v_0} $. 

By Bourgain, Gamburd and Sarnak \cite{BGS}, $L^2(\G_d\ba \bH^2)$
has a uniform spectral gap, that is, there exists ${\s}_0>0$ such that
if the second smallest eigenvalue of
the Laplacian in  $L^2(\G_d\ba \bH^2)$ is $s_1(d)(1-s_1(d))$, then $s_1 (d) +{\s}_0  <\delta $
for all square-free $d$ (note that $\delta(1-\delta)$ is the bottom of the spectrum
for all $\G_d$).

Set \begin{equation*} \Xi_{v_0}(\G, \Omega):=\frac{\kappa_{\iota^{-1}(\G)}}{\delta} 
\int_{k\in \Omega^{-1}} \frac{ d\nu_i^{\Gamma}(k(0))} {\|v_0k^{-1} \|^{\delta}} . \end{equation*}

Since the Patterson measure $\nu_i^\Gamma$ is normalized
so that $\phi_0^{\Gamma}(e)=|\nu_i^\Gamma|$ and $\|\phi_0^\G\|_2=1$,
we note that
$$\nu_i^{\Gamma_d}=\frac{1}{\sqrt{[\G:\G_d]}}\nu_i^\Gamma .$$
 Therefore
$\kappa_{\iota^{-1}(\G_d)}= \frac{1}{\sqrt{[\G:\G_d]}}\kappa_{\iota^{-1}(\G)}$ and
hence
$$\Xi_{v_0}(\G_d,\Omega)=\frac{\Xi_{v_0}(\G,\Omega)}{{[\G:\G_d]}}.$$

Hence Theorem \ref{ec} implies Theorem \ref{ec2m}.


\bibliographystyle{plain}

\end{document}